\documentclass{amsart}
\usepackage{amsmath}
\usepackage{amssymb}
\usepackage[all]{xy}
\usepackage{enumerate}
\usepackage{stmaryrd}

\theoremstyle{plain}
\newtheorem{thm}{Theorem}[section]
\newtheorem{lemma}[thm]{Lemma}

\newtheorem{prop}[thm]{Proposition}

\theoremstyle{definition}
\newtheorem{defn}[thm]{Definition}
\newtheorem{remark}[thm]{Remark}
\newtheorem{ex}[thm]{Example}

\begin{document}

\title{Exponentiation of Commuting nilpotent varieties}

\author[Paul Sobaje]
{Paul Sobaje}

\begin{abstract}

\sloppy{
Let $H$ be a linear algebraic group over an algebraically closed field of characteristic $p>0$.  We prove that any ``exponential map" for $H$ induces a bijection between the variety of $r$-tuples of commuting $[p]$-nilpotent elements in $Lie(H)$ and the variety of height $r$ infinitesimal one-parameter subgroups of $H$.  In particular, we show that for a connected reductive group $G$ in pretty good characteristic, there is a canonical exponential map for $G$ and hence a canonical bijection between the aforementioned varieties, answering in this case questions raised both implicitly and explicitly by Suslin, Friedlander, and Bendel.}

\end{abstract}

\maketitle

Let $H$ be a linear algebraic group over an algebraically closed field $k$ of characteristic $p>0$.  Let $H_{(r)}$ denote the $r$-th Frobenius kernel of $H$, and $\mathfrak{h}$ its Lie algebra.  There is a $[p]$-mapping on $\mathfrak{h}$ which sends $x \mapsto x^{[p]}$, and we set $\mathcal{N}_1(\mathfrak{h})$ to be the restricted nullcone of $\mathfrak{h}$, which consists of those $x$ for which $x^{[p]}=0$.  Let $\mathcal{C}_r(\mathcal{N}_1(\mathfrak{h}))$ denote the variety of $r$-tuples of commuting elements in $\mathcal{N}_1(\mathfrak{h})$, while $\text{Hom}_{gs/k}(\mathbb{G}_{a(r)},H)$ is the affine variety of height $r$ infinitesimal one-parameter subgroups of $H$ (it is the set of $k$-points of the affine scheme in \cite[Theorem 1.5]{SFB1}).  It was shown by Suslin, Friedlander, and Bendel in \cite{SFB1} and \cite{SFB2} that this last variety is homeomorphic to the cohomological variety of $H_{(r)}$, thereby establishing its importance within representation theory, specifically within the study of support varieties for modules.

Our aim in this paper is to provide a better understanding of the relationship between $\mathcal{C}_r(\mathcal{N}_1(\mathfrak{h}))$ and $\text{Hom}_{gs/k}(\mathbb{G}_{a(r)},H)$.  We give a definition of what we call an \textit{exponential map} for $H$, and show that any such map, if it exists, induces a bijection between these varieties.  In particular, we show that if $G$ is a connected reductive group over $k$ in \textit{pretty good} characteristic (see Section 1), then there is a canonical exponential map for $G$, and hence a canonical bijection between $\mathcal{C}_r(\mathcal{N}_1(\mathfrak{g}))$ and $\text{Hom}_{gs/k}(\mathbb{G}_{a(r)},G)$, extending results along these lines found in \cite{SFB1}, \cite{M}, and \cite{So2}.  We also give an example of a linear algebraic group $H$ and an $r > 1$ for which $\mathcal{C}_r(\mathcal{N}_1(\mathfrak{h}))$ and $\text{Hom}_{grp/k}(\mathbb{G}_{a(r)},H)$ are varieties which have different dimensions.

The importance of obtaining an explicit description of $\text{Hom}_{grp/k}(\mathbb{G}_{a(r)},H)$ is (from our point of view) primarily because it is a necessary step in extending support variety computations such as those in \cite{NPV}, \cite{DNP}, and \cite{So1}.  However, the results in this paper also have a new application to recent work by Friedlander \cite{F}, in which the author defines support varieties for rational $H$-modules (that is, support varieties directly defined for $H$ rather than for the Frobenius kernels of $H$).

This paper is organized as follows.  After recalling relevant concepts in Section 1, we give in Section 2 our definition of an exponential map on $\mathcal{N}_1(\mathfrak{h})$ and prove that it induces the variety bijection mentioned above.  Section 3 contains the strongest results which are available only in the case of connected reductive groups.  Our final section uses the work in Section 3 to give a simplified proof that one can achieve \textit{saturation} for a connected reductive group in good characteristic.  This notion is due to Serre, and involves assigning to each $p$-unipotent element $g \in G$ a one-parameter subgroup of $G$ whose image contains $g$, and which can be specified in some canonical way.  Seitz explored saturation in \cite{Sei} and proved that it can be achieved for any such $G$.

\section{Preliminaries and Notation}

We fix $k$ to be an algebraically closed field of characteristic $p>0$.  If $H$ is an affine algebraic group over $k$, then it is also an affine group scheme over $k$, and by abuse of notation we will use $H$ to denote both the scheme and the group of $k$-points of the scheme (as will be clear by the context).  Let $\mathcal{U}_1(H)$ denote the closed subset of $p$-unipotent elements in $H$ (i.e. the unipotent variety of $H$).  For $g,h \in H$, $X \in \mathfrak{h}$, we write $g \cdot h = ghg^{-1}$, while $g \cdot X$ denotes the adjoint action.  The centralizer of $g$ is $C_H(g)$, $C_H(X)$ is the stablizer of $X$ in $H$, $C_{\mathfrak{h}}(X)$ is the centralizer of $X$ in $\mathfrak{h}$, and $C_{\mathfrak{h}}(g)$ are those elements in $\mathfrak{h}$ which are stabilized by $g$.

For any affine group scheme $H$ (i.e. not necessarily coming from an algebraic group), we write $k[H]$ for its coordinate ring.  We denote by $Dist(H)$ its algebra of distributions (see \cite[I.7]{J2}), and by $\mathfrak{h}$ its Lie algebra.  Let $H_{(r)}$ denote the $r$-th Frobenius kernel of $H$ (see \cite[I.9]{J2} regarding how to define such a kernel without specifying an $\mathbb{F}_p$-structure on $H$).  We then have that $Dist(H_{(r)}) \subseteq Dist(H_{(r+1)})$, and $Dist(H) = \bigcup_{r \ge 1} Dist(H_{(r)})$.  We also recall that $Dist(H_{(1)})$ is isomorphic as a Hopf algebra to the restricted enveloping algebra of $\mathfrak{h}$ \cite[I.9.6(4)]{J2}.

If $\varphi$ is a homomorphism of affine group schemes from $H_1$ to $H_2$, then it induces a homomorphism of Hopf algebras from $Dist(H_1)$ to $Dist(H_2)$, and we will denote this map by $d\varphi$.  By abuse of notation we will also use $d\varphi$ in the more standard way to denote the differential of $\varphi$, that is the induced map from $\mathfrak{h}_1$ to $\mathfrak{h}_2$.  In fact, this map on Lie algebras can really be viewed as a restriction of the map on distribution algebras by the comments above.

The additive group $\mathbb{G}_a$ has coordinate algebra $k[\mathbb{G}_a] \cong k[t]$, and $Dist(\mathbb{G}_a)$ is spanned by the elements $\frac{d}{dt}^{(j)}$, where

$$\frac{d}{dt}^{(j)}(t^i) = \delta_{ij}.$$

This notation is not standard.  For example, Jantzen denotes the element $\frac{d}{dt}^{(j)}$ as $\gamma_j$ in \cite[I.7.8]{J2}.  If we set $u_j = \frac{d}{dt}^{(p^j)}$, and if $m$ is an integer with $p$-adic expansion $$m = m_0 + m_1p + \cdots + m_qp^q, \quad 0 \le m_i < p,$$ then

$$\frac{d}{dt}^{(m)} = \frac{u_0^{m_0}\cdots u_q^{m_q}}{m_0!\cdots m_q!}$$

Therefore $Dist(\mathbb{G}_a)$ is generated as an algebra over $k$ by the set $\{u_j\}_{j \ge 0}$, while $Dist(\mathbb{G}_{a(r)})$ is generated by the subset where $j<r$.  Also, for any affine group scheme $H$, a homomorphism from $\mathbb{G}_{a(r)}$ to $H$ is equivalent to a Hopf algebra homomorphism from $Dist(\mathbb{G}_{a(r)})$ to $Dist(H)$, this latter homomorphism being determined by the images of the elements $u_j$.

Throughout this paper $G$ will always denote a connected reductive algebraic group over $k$.  We fix a maximal torus $T$ of $G$, and denote by $\Phi$ the root system of $G$ with respect to $T$.  We choose a set of simple roots $\Pi \subseteq \Phi$ which determines the set of positive roots $\Phi^+$.  Let $B$ denote the Borel subgroup of $G$ containing the root subgroups corresponding to every positive root.  For any $J \subseteq \Pi$, let $P_J$ be the corresponding parabolic subgroup of $G$.  We define $\Phi_J = \Phi \cap \mathbb{Z}J$, and $\Phi_J^+ = \Phi^+ \cap \Phi_J$.

The prime $p$ is said to be good for $G$ if it is good for $\Phi$.  Specifically, this means that $p > 2$ if $\Phi$ has a component of type $B,C,$ or $D$; $p > 3$ if $\Phi$ has a component of type $E_6, E_7, F_4$ or $G_2$; and $p > 5$ if $E_8$ is a component of $\Phi$.  Equivalently, $p$ is good for $\Phi$ if and only if the $p$-torsion of $\mathbb{Z}\Phi/\mathbb{Z}\Phi^{\prime}$ vanishes for every subset $\Phi^{\prime} \subseteq \Phi$ \cite[Lemma 2.10]{H}.

We will need the following related notion which is due to Herpel.  Let $X$ be the character group of $T$, $Y$ the cocharacter group, and $\Phi^{\vee}$ the set of coroots, so that the quadruple $(X,\Phi,Y,\Phi^{\vee})$ is the root datum of $G$ with respect to $T$.  Then $p$ is said to be pretty good for $G$ if the groups $X/\mathbb{Z}\Phi^{\prime}$ and $Y/\mathbb{Z}{\Phi^{\prime}}^{\vee}$ have no $p$-torsion for all subsets $\Phi^{\prime} \subseteq \Phi$ \cite[Definition 2.11]{H}.

For each $\alpha \in \Phi$, fix a root homomorphism $\varphi_{\alpha}: \mathbb{G}_a \rightarrow G$.  We then set $e_{\alpha} = d\varphi_{\alpha}(\frac{d}{dt}^{(1)}) \in \mathfrak{g}$.

\section{Exponential Maps and Infinitesimal One-parameter Subgroups}

Let $H$ be linear algebraic group over $k$.  It follows from \cite[I.8.4]{J2} that a homomorphism of affine group schemes from $\mathbb{G}_{a(r)}$ to $H_{(r)}$ is equivalent to a homomorphism of Hopf algebras from $Dist(\mathbb{G}_{a(r)})$ to $Dist(H_{(r)}) \subseteq Dist(H)$.  Since $Dist(\mathbb{G}_{a(r)})$ is finite dimensional over $k$, it also follows that any Hopf algebra homormorphism from $Dist(\mathbb{G}_{a(r)})$ to $Dist(H)$ factors through $Dist(H_{(s)})$ for some $s$, hence comes from a homomorphism from $\mathbb{G}_{a(r)}$ to $H_{(s)} \subseteq H$.  But any homomorphism $\varphi$ from $\mathbb{G}_{a(r)}$ to $H$ must factor through $H_{(r)}$.  Indeed, if $\varphi^*$ is the induced map from $k[H]$ to $k[\mathbb{G}_{a(r)}]$ and $x$ is in the augmentation ideal of $k[H]$, then $\varphi^*(x^{p^r}) = \varphi^*(x)^{p^r} = 0$, hence $\varphi$ factors through the natural projection $k[H] \rightarrow k[H_{(r)}]$.  In conclusion we see that Hopf algebra homomorphisms from $Dist(\mathbb{G}_{a(r)})$ to $Dist(H)$ are equivalent to group scheme homomorphisms from $\mathbb{G}_{a(r)}$ to $H$.

For each $X \in \mathcal{N}_1(\mathfrak{h})$ there is a homomorphism of Hopf algebras from $Dist(\mathbb{G}_{a(1)})$ to $Dist(H)$ which sends $u_0$ to $X$.  By the preceding discussion there is then an identification between $\mathcal{N}_1(\mathfrak{h})$ and $\text{Hom}_{grp/k}(\mathbb{G}_{a(1)},H)$ for all $H$.  When $r > 1$ however, it is not immediately clear how a homomorphism from $Dist(\mathbb{G}_{a(r)})$ to $Dist(H)$ determines (and is determined by) an element in $\mathcal{C}_r(\mathcal{N}_1(\mathfrak{h}))$.  In fact, as we show at the end of the section, there are algebraic groups for which $\text{Hom}_{grp/k}(\mathbb{G}_{a(r)},H) \ncong C_r(\mathcal{N}_1(\mathfrak{h}))$.  We will therefore need to restrict our attention to those $H$ for which an \textit{exponential map} exists (defined below).

We start with an important lemma which is in fact valid for more general affine group schemes over $k$ (that is, not only those for which $k[H]$ is reduced).

\begin{lemma}\label{equal}
Let $H$ be an affine group scheme such that $k[H]$ is a finitely generated $k$-algebra, and let $\varphi_1, \varphi_2 \in \textup{Hom}_{grp/k}(\mathbb{G}_{a(r)}, H)$. Let $0<m<r$, and suppose that $d\varphi_1(u_i) = d\varphi_2(u_i)$, for all $i < m$. Then $d\varphi_1(u_m) - d\varphi_2(u_m) \in \mathfrak{h}$.
\end{lemma}

\begin{proof}
Let $\Delta_H$ and $\Delta_{\mathbb{G}_a}$ denote the comultiplication maps on $Dist(H)$ and $Dist(\mathbb{G}_a)$ respectively.  We have then that

\vspace{0.06in}
\begin{center} $\Delta_H(d\varphi_1(u_m) - d\varphi_2(u_m)) = d\varphi_1 \otimes d\varphi_1(\Delta_{\mathbb{G}_a}(u_m)) - d\varphi_2 \otimes d\varphi_2(\Delta_{\mathbb{G}_a}(u_m)),$ \end{center} 
\vspace{0.06in}

\noindent and

$$\Delta_{\mathbb{G}_a}(u_m) = u_m \otimes 1 + 1 \otimes u_m + \sum y_1 \otimes y_2,$$

\vspace{0.2in}

\noindent with each $y_i$ contained in $Dist(\mathbb{G}_{a(m)})$. By assumption, $d\varphi_1(u_i) = d\varphi_2(u_i)$ for all $i < m$, hence $d\varphi_1$ and $d\varphi_2$ agree on $Dist(\mathbb{G}_{a(m)})$. From this it follows that

\vspace{0.06in}
\begin{center} $\Delta_H(d\varphi_1(u_m) - d\varphi_2(u_m)) = \left( d\varphi_1(u_m) - d\varphi_2(u_m)\right) \otimes 1 + 1 \otimes \left( d\varphi_1(u_m) - d\varphi_2(u_m) \right).$ \end{center} 
\vspace{0.06in}

Thus, the comultiplication of $d\varphi_1(u_m) - d\varphi_2(u_m)$ is primitive. By \cite[\S 3.18]{T}, the primitive elements in $Dist(H)$ all lie in $\mathfrak{h}$, proving the claim.
\end{proof}

\begin{ex}
Consider $SL_2$ in its natural representation on $k^2$ and let $X$ be a fixed nilpotent element in $\text{End}_k(k^2)$.  Define one-parameter subgroups $\varphi_1, \varphi_2, \varphi_3$ of $SL_2$ by $$\varphi_1(s) = 1 + sX,\; \varphi_2(s) = 1+s^pX, \; \varphi_3(s) = \varphi_1(s)\varphi_2(s).$$

For each $i \ge 1$, let $X^{(i)} = d\varphi_1(\frac{d}{dt}^{(i)})$.  Note that $X^{(1)} = X$.  Comparing differentials, we see that, by definition, $d\varphi_1(u_1) = X^{(p)}$, while $d\varphi_2(u_1) = X^{(1)}$ and $\varphi_3(u_1) = X^{(p)} + X^{(1)}$.  We further have that $d\varphi_1(u_0) = d\varphi_3(u_0)$, and that $d\varphi_3(u_1) - d\varphi_1(u_1) = X^{(1)} \in \mathfrak{sl}_2$.
\end{ex}

\begin{defn}\label{map}
Let $H$ be an affine algebraic group over $k$.  We say that $H$ has an \textit{exponential map} if there exists an $H$-equivariant injective map

\vspace{0.06in}
\begin{center} $\mathcal{E}: \mathcal{N}_1(\mathfrak{h}) \rightarrow H$ \end{center}
\vspace{0.06in}

\noindent such that for each $X \in \mathcal{N}_1(\mathfrak{h})$:

\begin{enumerate}

\item The map $\mathcal{E}_X$, which sends $s \in \mathbb{G}_a$ to $\mathcal{E}(sX)$, is a one-parameter additive subgroup of $H$ whose differential maps $u_0 \in Dist(\mathbb{G}_a)$ to $X$.

\item The one-parameter subgroup $\mathcal{E}_X(\mathbb{G}_a)$ acts trivially in the adjoint action on $C_{\mathfrak{h}}(X) \subseteq \mathfrak{h}$.

\end{enumerate}

\end{defn}

For example, the truncated exponential series defines such a map for $GL_n$, and also for any subgroup of $GL_n$ which is of exponential type \cite[Lemma 1.7]{SFB1}.

We observe that the conditions listed above imply the following important property of an exponential map:

\begin{lemma}
If $\mathcal{E}$ is an exponential map for $H$, and if $X,Y \in \mathcal{N}_1(\mathfrak{h})$ are such that $[X,Y] = 0$, then $\mathcal{E}(X)$ and $\mathcal{E}(Y)$ commute in $H$.
\end{lemma}

\begin{proof}
If $[X,Y] = 0$, then $Y \in C_{\mathfrak{h}}(X)$.  By assumption $\mathcal{E}(X)$ acts trivially on $C_{\mathfrak{h}}(X)$, hence $\mathcal{E}(X) \in C_H(Y)$.  By the $H$-equivariance of $\mathcal{E}$ it follows that $\mathcal{E}(X) \in C_H(\mathcal{E}(Y))$.
\end{proof}

Suppose now that $H$ has an exponential map $\mathcal{E}$.  Let $F$ denote the standard Frobenius morphism on $\mathbb{G}_a$, defined by $F(s) = s^p$.  For any $X \in \mathcal{N}_1(\mathfrak{h})$ and any $i \ge 0$, we denote by $\mathcal{E}_{X}^{(i)}$ the one parameter subgroup of $H$ given by $\mathcal{E}_{X} \circ F^i$.  For any $r \ge 1$, and any $r$-tuple $(X_0, X_1, \cdots, X_{r-1})$ of pairwise commuting elements in $\mathcal{N}_1(\mathfrak{h})$ we then obtain a one-parameter subgroup of $H$

\vspace{0.06in}
\begin{center} $\mathcal{E}_{X_0}\mathcal{E}_{X_1}^{(1)}\cdots \mathcal{E}_{X_{r-1}}^{(r-1)}(s) = \mathcal{E}(sX_0)\mathcal{E}(s^pX_1)\cdots \mathcal{E}(s^{p^{r-1}}X_{r-1})$. \end{center} 
\vspace{0.06in}

This one-parameter subgroup can be restricted to $\mathbb{G}_{a(r)}$, defining a map $$\mathcal{E}_*: \mathcal{C}_r(\mathcal{N}_1(\mathfrak{h})) \rightarrow \text{Hom}_{grp/k}(\mathbb{G}_{a(r)},H).$$  Our next result is the main theorem in this section and shows that $\mathcal{E}_*$ is a bijection.

\bigskip
\begin{thm}\label{description}

Let $\mathcal{E}$ be an exponential map for $H$.  Then the map $$\mathcal{E}_*: \mathcal{C}_r(\mathcal{N}_1(\mathfrak{h})) \rightarrow \text{Hom}_{grp/k}(\mathbb{G}_{a(r)},H),$$ defined by $$\mathcal{E}_*(X_0,\ldots,X_{r-1}) = \mathcal{E}_{X_0}\mathcal{E}_{X_1}^{(1)}\cdots \mathcal{E}_{X_{r-1}}^{(r-1)}\mid_{\mathbb{G}_{a(r)}},$$ is a bijection.

\end{thm}

\begin{proof}
To show that $\mathcal{E}_*$ is injective we reproduce the argument given in the proof of \cite[Theorem 9.6]{M}. Suppose that $\mathcal{E}_{X_0}\mathcal{E}_{X_1}^{(1)} \cdots \mathcal{E}_{X_{r-1}}^{(r-1)} = \mathcal{E}_{Y_0}\mathcal{E}_{Y_1}^{(1)} \cdots \mathcal{E}_{Y_{r-1}}^{(r-1)}$. Because these homomorphisms agree when restricted to $\mathbb{G}_{a(1)}$, we must have that $X_0 = Y_0$. We can therefore multiply each homomorphism by $\mathcal{E}_{-X_0}$, and since $\mathcal{E}_{X_0}\mathcal{E}_{-X_0} = Id.$, it follows that $\mathcal{E}_{X_1}^{(1)} \cdots \mathcal{E}_{X_{r-1}}^{(r-1)} = \mathcal{E}_{Y_1}^{(1)} \cdots \mathcal{E}_{Y_{r-1}}^{(r-1)}$.  This says that

\vspace{0.05in}
\begin{center}$(\mathcal{E}_{X_1} \cdots \mathcal{E}_{X_{r-1}}^{(r-2)}) \circ F = (\mathcal{E}_{Y_1} \cdots \mathcal{E}_{Y_{r-1}}^{(r-2)}) \circ F$\end{center}
\vspace{0.05in}

\noindent when restricted to $\mathbb{G}_{a(r)}$.  As $F(\mathbb{G}_{a(r)}) = \mathbb{G}_{a(r-1)}$, we can proceed by induction to see that $X_i = Y_i$ for all $i$.

To prove that $\mathcal{E}_*$ is surjective, suppose that $\phi: \mathbb{G}_{a(r)} \rightarrow H$ is an infinitesimal one-parameter subgroup of $H$. We define the following sequence of elements in $Dist(H)$:

\vspace{0.05in}
\begin{center}$X_0 = d\phi(u_0); \qquad X_i = d\phi(u_i) - d(\mathcal{E}_{X_0}\mathcal{E}_{X_1}^{(1)}\cdots \mathcal{E}_{X_{i-1}}^{(i-1)})(u_i)$\end{center}
\vspace{0.05in}

We will prove by induction that this defines a sequence of commuting elements in $\mathcal{N}_1(\mathfrak{h})$, and further that for any $i \le r$, the homomorphism
$\mathcal{E}_{X_0}\mathcal{E}_{X_1}^{(1)}\cdots\mathcal{E}_{X_{i-1}}^{(i-1)}$ is the same as $\phi$ when restricted to $\mathbb{G}_{a(i)}$.

First we consider the case of $i=1$.   Because $u_0^p = 0$, we have that $X_0 \in \mathcal{N}_1(\mathfrak{h})$, and that $\mathcal{E}_{X_0}$ and $\phi$ are equal on $\mathbb{G}_{a(1)}$.  By Lemma \ref{equal}, the element $X_1 = d\phi(u_1) - d\mathcal{E}_{X_0}(u_1)$ lies in $\mathfrak{h}$. To see that $X_1^p=0$, we note first that both $d\phi(u_1)$ and $d\mathcal{E}_{X_0}(u_1)$ commute with $X_0$, as $u_0$ and $u_1$ commute in $Dist(\mathbb{G}_a)$ and $X_0 = d\phi(u_o) = d\mathcal{E}_{X_0}(u_0)$. Thus $X_1 \in C_{\mathfrak{h}}(X_0)$.  By assumption the one parameter subgroup of $H$ given by $\mathcal{E}_{X_0}$ acts trivially in the adjoint action on $X_1$, so that $d\mathcal{E}_{X_0}(u_j)$ commutes with $X_1$ for all $j$. But $[d\mathcal{E}_{X_0}(u_1),X_1] = 0$ if and only if $[d\mathcal{E}_{X_0}(u_1),d\phi(u_1)] = 0$, therefore

\vspace{0.05in}
\begin{center}$X_1^p = \left(d\phi(u_1) - d\mathcal{E}_{X_0}(u_1)\right)^p = d\phi(u_1)^p - d\mathcal{E}_{X_0}(u_1)^p = 0$\end{center}
\vspace{0.05in}

\noindent so that $X_1 \in \mathcal{N}_1(\mathfrak{h})$.

Suppose now that $i \ge 2$ and that $X_0, \ldots X_{i-1}$ have been chosen as above so that
$\mathcal{E}_{X_0}\cdots \mathcal{E}_{X_{i-2}}^{(i-2)}$ is the same as $\phi$ when restricted to $\mathbb{G}_{a(i-1)}$, and that $X_{i-1} =
d\phi(u_{i-1}) - d\text{exp}_{X_0}\cdots\text{exp}_{X_{i-2}}^{(i-2)}(u_{i-1})$ is $p$-nilpotent.  To complete the inductive step we must show that
$\mathcal{E}_{X_0}\cdots\mathcal{E}_{X_{i-1}}^{(i-1)}$ is equal to $\phi$ on $\mathbb{G}_{a(i)}$, and that $d\mathcal{E}_{X_0}\cdots\mathcal{E}_{X_{i-1}}^{(i-1)}(u_{i})$ commutes with $d\phi(u_{i})$ so that $X_{i}^p=0$.

To prove the first claim, since $\mathcal{E}_{X_0}\cdots\mathcal{E}_{X_{i-1}}^{(i-1)}$ is equal to $\phi$ on $\mathbb{G}_{a(i-1)}$, we only need to show that $d(\mathcal{E}_{X_0}\cdots\mathcal{E}_{X_{i-1}}^{(i-1)})(u_{i-1}) = d\phi(u_{i-1})$.  Following the general argument presented in \cite[Section 2]{So1}, we first note that we can factor $\mathcal{E}_{X_0}\cdots \mathcal{E}_{X_{i-1}}^{(i-1)}$ as

$$\mathbb{G}_a \xrightarrow{\delta} \mathbb{G}_a \times \mathbb{G}_a \xrightarrow{\mathcal{E}_{X_0}\cdots\mathcal{E}_{X_{i-2}}^{(i-2)} \, \times \, \mathcal{E}_{X_{i-1}}^{(i-1)}} H
\times H \xrightarrow{m} H$$

\vspace{0.05in}
\noindent where $m$ is the multiplication map on $H$ and $\delta$ is the diagonal map $s \mapsto (s,s)$. By the reasoning given in the proof of \cite[Proposition 2.3]{So1}, it
then follows that

\begin{eqnarray*}
d(\mathcal{E}_{X_0}\cdots\mathcal{E}_{X_{i-1}}^{(i-1)})(u_{i-1}) & = & d(\mathcal{E}_{X_0}\cdots\mathcal{E}_{X_{i-2}}^{(i-2)})(u_{i-1}) + d\mathcal{E}_{X_{i-1}}(u_0)\\
& = & d(\mathcal{E}_{X_0}\cdots\mathcal{E}_{X_{i-2}}^{(i-2)})(u_{i-1}) + X_{i-1}\\
& = & d\phi(u_{i-1})\\
\end{eqnarray*}

Therefore these maps agree on $Dist(\mathbb{G}_{a(i)})$.  For the second claim of the inductive step, we see that $X_{i}$ commutes with $X_0$ for the same reason that $X_1$ did above. Hence $X_{i}$ commutes with $d\mathcal{E}_{X_0}(u_1)$, and so commutes with $d\phi(u_1) - d\mathcal{E}_{X_0}(u_1) = X_1$.  Proceeding in this way, we see that $X_{i}$ commutes with $X_j$ for $j \le i$, hence with $\mathcal{E}_{X_j}(u_{\ell})$ for all $\ell$, and therefore $X_i$ commutes with $d\mathcal{E}_{X_0}\cdots \mathcal{E}_{X_{i-1}}^{(i-1)}(u_{i})$.  This implies that $X_{i}$ is $p$-nilpotent, completing the inductive step.

\end{proof}

We conclude this section with an example which demonstrates that it is not always true that there is an identification between $\mathcal{C}_r(\mathcal{N}_1(\mathfrak{h}))$ and $\text{Hom}_{grp/k}(\mathbb{G}_{a(r)},H)$.

\begin{ex}

Assume that $p > 2$.  Let $H$ be the ``fake Heisenberg group" described in \cite{BD}, which is a 2-dimensional connected nonabelian unipotent algebraic group over $k$ (and cleverly named, as the ``true" Heisenberg group is the smallest nonabelian connected unipotent group in characteristic $0$).  It is isomorphic as a variety to $\mathbb{A}^2$, and has group structure given by

$$(a,b) \cdot (c,d) = \left(a+c,b+d+\frac{1}{2}(a^pc - ac^p)\right)$$

We have then that $k[H] \cong k[X,Y]$, with Hopf algebra structure given by $$\Delta(X) = X \otimes 1 + 1 \otimes X, \; \Delta(Y) = Y \otimes 1 + 1 \otimes Y + \frac{1}{2}(X^p \otimes X - X \otimes X^p).$$

Let $\epsilon$ be the counit on $k[H]$.  As $(0,0)$ is the identity element in $H$, we have that $\epsilon(X) = 0 = \epsilon(Y)$, so that $$k[H_{(r)}] \cong k[X,Y]/(X^{p^r},Y^{p^r}).$$  When $r = 1$, the comultiplication of the image of $Y$ in $k[H_{(1)}]$ is primitive, thus $H_{(1)} \cong \mathbb{G}_{a(1)} \times \mathbb{G}_{a(1)}$.  We see then that $\mathfrak{h}$ is two-dimensional with trivial bracket, so that $$\mathcal{C}_r(\mathcal{N}_1(\mathfrak{h})) \cong \mathbb{A}^{2r}.$$  On the other hand, the variety $\text{Hom}_{grp/k}(\mathbb{G}_{a(r)},H)$ can be computed by looking directly at all Hopf algebra maps from $k[H]$ to $k[t]/(t^{p^r})$.  Let $f$ be such a map.  We note that as $X$ is primitive, it must be sent to a primitive element in $k[t]/(t^{p^r})$.  Thus

\begin{equation}\label{Ximage}
f(X) = a_0t + a_1t^p + \cdots + a_{r-1}t^{p^{r-1}}
\end{equation}

Now consider $f(Y)$.  Let $\tau$ be the twist map $\tau(x \otimes y) = y \otimes x$.  As $k[t]/(t^{p^r})$ is cocommutative, and $f$ is assumed to be a Hopf algebra map, we must have that $$\tau \circ (f \otimes f) \circ \Delta(Y) = (f \otimes f) \circ \Delta(Y).$$  Therefore

\begin{eqnarray*}
& 1 \otimes f(Y) + f(Y) \otimes 1 +  \frac{1}{2}\left(f(X) \otimes f(X^p) - f(X^p) \otimes f(X) \right)\\
= & f(Y) \otimes 1 + 1 \otimes f(Y) + \frac{1}{2}\left(f(X^p) \otimes f(X) - f(X) \otimes f(X^p)\right)\\
\end{eqnarray*}

From this we deduce that $f(X^p) \otimes f(X) = f(X) \otimes f(X^p)$.  As the elements $t^i \otimes t^j, \quad 0 \le i,j < p^r$, are linearly independent in $k[t]/(t^{p^r}) \otimes k[t]/(t^{p^r})$, the above equality can only hold when $f(X^p) = 0$.  This of course implies that in (\ref{Ximage}), only $a_{r-1} \ne 0$.   With that established, we then see that $f(Y)$ can be any primitive element in $k[t]/(t^{p^r})$.

In conclusion, any Hopf algebra map $f:k[H] \rightarrow k[t]/(t^{p^r})$ is given by $$f(X) = a_{r-1}t^{p^{r-1}}, \quad f(Y) = b_0t + b_1t^p + \cdots + b_{r-1}t^{p^{r-1}}$$ where $a_{r-1},b_i \in k$.  Therefore $$\text{Hom}_{grp/k}(\mathbb{G}_{a(r)},H) \cong \mathbb{A}^{r+1}.$$  If $r>1$, we see then that  $\text{Hom}_{grp/k}(\mathbb{G}_{a(r)},H) \ncong \mathcal{C}_r(\mathcal{N}_1(\mathfrak{h}))$.

\end{ex}

\section{Exponential Maps for Reductive Groups}

Suppose now that $G$ is a connected reductive group over $k$ and $p$ is good for $G$.  In this section we will prove the existence of a canonical $G$-equivariant bijection from $\mathcal{N}_1(\mathfrak{g})$ to $\mathcal{U}_1(G)$ which is an exponential map in pretty good characteristic (in fact, this canonical bijection is also an exponential map if $G$ is the derived subgroup of a reductive group in pretty good characteristic).  The existence of this exponential map when $p \ge h$ (and a few other cases) can be deduced from a result by Carlson, Lin, and Nakano \cite[Theorem 3]{CLN}.

It was proved in \cite{NPV} that if $G$ is connected reductive in good characteristic, then there exists a parabolic subgroup $P_J \le G$ with unipotent radical $U_J$ for which

\vspace{0.05in}
\begin{center} $G \cdot \mathfrak{u}_J = \mathcal{N}_1(\mathfrak{g}).$\end{center}
\vspace{0.05in}

\noindent (Of course, if there is one such parabolic then there are many, however up to conjugation we may restrict our attention to standard parabolics).  It was further shown in \cite[Theorem 2]{CLN} that one may choose $P_J$ so that the nilpotence class of $\mathfrak{u}_J$ is less than $p$, and

\vspace{0.05in}
\begin{center}$G \cdot U_J = \mathcal{U}_1(G).$\end{center}
\vspace{0.05in}

We note that this latter result was stated for simple groups in good characteristic, but is also true for connected reductive groups.  Indeed, if $G$ is a connected reductive group, then by embedding $G$ into some $GL_n$, we see that every $p$-unipotent element $u \in G$ is, up to conjugation, in $G \,\cap\, U_n$, where $U_n$ is the group of strictly upper-triangular matrices of $GL_n$.  Hence $u$ is in a maximal unipotent subgroup of $G$, and hence is in $G^{\prime}$, so $\mathcal{U}_1(G) = \mathcal{U}_1(G^{\prime})$.  On the other hand, by \cite[Proposition 14.26]{B}, every nilpotent element in $\mathfrak{g}$ is in the Lie algebra of some maximal unipotent subgroup, hence is an element in $\mathfrak{g}^{\prime}$.  Thus $\mathcal{N}_1(\mathfrak{g}) = \mathcal{N}_1(\mathfrak{g}^{\prime})$.  We may therefore assume that $G$ is a connected semisimple group in good characteristic.  Let $G_{sc}$ denote the semisimple simply-connected group having the same root system as $G$.  Then $G_{sc}$ is the direct product of simple, simply-connected groups in good characteristic, hence the results above hold for $G_{sc}$.  But the isogeny $G_{sc} \rightarrow G$ induces bijections (though not necessarily isomorphisms) $\mathcal{U}_1(G_{sc}) \rightarrow \mathcal{U}_1(G)$ and $\mathcal{N}_1(\mathfrak{g}_{sc}) \rightarrow \mathcal{N}_1(\mathfrak{g})$ .  Hence the result holds for $G$.

When the nilpotence class of $\mathfrak{u}_J$ is less than $p$, we know by a result due to Serre that there exists a unique exponential map on $\mathfrak{u}_J$ which arises from base-changing the exponential isomorphism in characteristic $0$ (see \cite[\S 8]{M} for a nice account of this). Before stating this result precisely, we recall that if $\mathfrak{u}_J$ has nilpotence class less than $p$ then it can be made into a group via the Baker-Campbell-Hausdorff formula.  Henceforth, any reference to a Lie algebra (of nilpotence class less than $p$) as an algebraic group will be assuming this group structure.

\begin{prop}\label{Serre}\cite[Proposition 5.3]{Sei}
Let $P$ be a parabolic subgroup of $G$ with unipotent radical $U$, and suppose that $U$ has nilpotence class less than $p$.  Then there is a unique $P$-equivariant isomorphism of algebraic groups

\vspace{0.06in}
\begin{center} $\varepsilon_P: \mathfrak{u} \xrightarrow{\sim} U$ \end{center}
\vspace{0.06in}

\noindent with tangent map the identity on $\mathfrak{u}$.  Further, if for some $\alpha \in \Phi$, $\varphi_{\alpha}$ is a root homomorphism of $G$ (with respect to $T$) factoring through $U$, then $\varepsilon_P(se_{\alpha}) = \varphi_{\alpha}(s)$ for all $s \in k$. 
\end{prop}

\begin{remark}
The last statement of the proposition regarding root subgroups is established in the proof of \cite[Proposition 5.2]{Sei}.  Note that it applies to any root with respect to some maximal torus $T$, and also any chosen root homomorphism (for the choice of $e_{\alpha}$ corresponds to the chosen root homomorphism).
\end{remark}

\bigskip
It was the idea of Carlson-Lin-Nakano in \cite{CLN} to use these results above to obtain a type of exponential map on the restricted nullcone by extending $\varepsilon_{P_J}$. Indeed this was achieved in Theorem 3 of \emph{loc. cit.} under the assumption that $\mathcal{N}_1(\mathfrak{g})$ is a normal variety.  This assumption is known to hold in particular cases, for example if $p \ge h$ (in which case the entire nullcone is restricted), and also for simple groups of type $A$.

We now show that $\varepsilon_{P_J}$ always extends to a well-defined map on $\mathcal{N}_1(\mathfrak{g})$, regardless of whether or not this variety is normal.  Furthermore, we show that this extended map is independent of the choice of $J \subseteq \Pi$.  As noted in \cite{CLNP}, the choice of $J$ for which $G \cdot \mathfrak{u}_J = \mathcal{N}_1(\mathfrak{g})$ is not necessarily unique, and the exponential map in \cite[Theorem 3]{CLN} is seemingly dependent this choice (of course in the case that $p \ge h$, this is not an issue).

\begin{prop}\label{choice}

Let $I, J \subseteq \Pi$ be such that both $U_I$ and $U_J$ are of nilpotence class less than $p$.

\begin{enumerate}

\item If $g \in G$ and $x \in \mathfrak{u}_J$ are such that $g \cdot x \in \mathfrak{u}_J$, then $\varepsilon_{P_J}(g \cdot x) = g \cdot \varepsilon_{P_J}(x).$

\item If $g \in G$, $x \in \mathfrak{u}_J$, and $y \in \mathfrak{u}_I$ are such that $y = g \cdot x$, then $\varepsilon_{P_I}(y) = g \cdot \varepsilon_{P_J}(x)$.

\end{enumerate}

\end{prop}

\begin{proof}
(1) Suppose first that $g \in N_G(T)$. Conjugation by $g$ sends $P_J$ to some other parabolic subgroup $P$ with unipotent radical $U$.  This conjugation action also defines algebraic group isomorphisms $U_J \xrightarrow{\sim} U$ and $\mathfrak{u}_J \xrightarrow{\sim} \mathfrak{u}$.  Let $\varepsilon: \mathfrak{u} \rightarrow U$ be the map given by $\varepsilon(y) = g \cdot \varepsilon_{P_J}(g^{-1} \cdot y)$.  Then by the preceding remarks we see that $\varepsilon$ is an isomorphism of algebraic groups from $\mathfrak{u}$ to $U$ which is easily seen to be $P$-equivariant.  Furthermore, it is also clear that the tangent map of $\varepsilon$ is the identity on $\mathfrak{u}$, hence that $\varepsilon$ is the unique map $\varepsilon_P$ specified in Proposition \ref{Serre}.

We observe that the condition that both $x \in \mathfrak{u}_J$ and $g \cdot x \in \mathfrak{u}_J$ is equivalent to the condition that $g \cdot x \in \mathfrak{u}_J \cap \mathfrak{u}$.  Therefore, proving that $\varepsilon_{P_J}(g \cdot x) = g \cdot \varepsilon_{P_J}(x)$ for all such $x$ is equivalent to showing that $\varepsilon_{P_J}$ agrees with $\varepsilon_P$ on $\mathfrak{u}_J \cap \mathfrak{u}$.  Indeed, if this is so then we would have from above that $$\varepsilon_{P_J}(g \cdot x) = \varepsilon_P(g \cdot x) = g \cdot \varepsilon_{P_J}(x).$$

Therefore we will show that $\varepsilon_{P_J}$ agrees with $\varepsilon_P$ on $\mathfrak{u}_J \cap \mathfrak{u}$.  Now, since $g$ acts on $T$ it acts on $\Phi$, and if $\alpha \in \Phi$ we will write $g \cdot \alpha$ for this action.  We also have that since $T$ normalizes both $\mathfrak{u}_J$ and $\mathfrak{u}$ it normalizes their intersection, and so $\mathfrak{u}_J \cap \mathfrak{u}$ has a basis consisting of certain $e_{\alpha}$.  In particular it is those $e_{\alpha}$ for which both $\alpha$ and $g^{-1} \cdot \alpha$ are in $(\Phi^+ \backslash \Phi_J^+)$ (this is observed for $P_J=B$ in \cite[I.1.3(b)]{SS}).  In a similar way $U_J \cap U$ is generated by the corresponding root subgroups.  Both $\varepsilon_P$ and $\varepsilon_{P_J}$ then restrict to group isomorphisms between $\mathfrak{u}_J \cap \, \mathfrak{u}$ and $U_J \cap \, U$.  We also have that $\varepsilon_P(se_{\alpha}) = \varphi_{\alpha}(s)= \varepsilon_{P_J}(se_{\alpha})$ for all $e_{\alpha} \in \mathfrak{u}_J \cap \, \mathfrak{u}$ and $s \in k$.  This then says that $\varepsilon_{P_J}$ and $\varepsilon_P$ are group isomorphisms which agree on a set of group generators for $\mathfrak{u}_J \cap \, \mathfrak{u}$, hence they agree on all of $\mathfrak{u}_J \cap \, \mathfrak{u}$.

We have therefore shown that if $g \in N_G(T)$ and $x \in \mathfrak{u}_J$ are such that $g \cdot x \in \mathfrak{u}_J$, then $\varepsilon_{P_J}(g \cdot x) = g \cdot \varepsilon_{P_J}(x)$.  Now let $g$ be any element in $G$ for which both $x$ and $g \cdot x$ are in $\mathfrak{u}_J$.  By the Bruhat decomposition of $G$, we know that $g = b_1nb_2$, where $b_1,b_2 \in B$ and $n \in N_G(T)$ (and recall that $B \le P_J$). We have then that $b_1nb_2 \cdot x \in \mathfrak{u}_J$ if and only if $nb_2 \cdot x \in \mathfrak{u}_J$. We also know that $b_2 \cdot x \in \mathfrak{u}_J$.  By the $P_J$-equivariance of $\varepsilon_{P_J}$ we have that $\varepsilon_{P_J}(b_2 \cdot x) = b_2 \cdot \varepsilon_{P_J}(x)$ and $\varepsilon_{P_J}(b_1 \cdot (nb_1 \cdot x)) = b_1 \cdot \varepsilon_{P_J}(nb_2 \cdot x)$.  Combining these with the above result for $n$ then establishes the first claim.

The proof of (2) is similar. Given $g \cdot x = y \in \mathfrak{u}_I$ for some $x \in \mathfrak{u}_J$, we again can write $g = b_1nb_2$ as above. Replacing $x$ with $b_2 \cdot x \in \mathfrak{u}_J$ and $y$ with $b_1^{-1} \cdot y \in \mathfrak{u}_I$, the $B$-equivariance of both $\varepsilon_{P_J}$ and $\varepsilon_{P_I}$ allow us to reduce to the case that $y = n \cdot x$ for some $n \in N_G(T)$. But then an argument nearly identical to that used in the proof of (1) will also hold here to show that $\varepsilon_{P_I}(n \cdot x) = n \cdot \varepsilon_{P_J}(x)$.
\end{proof}

\begin{thm}\label{exp}
Let $G$ be a connected reductive group over $k$, and assume that $p$ is good for $G$.  Then there is a unique $G$-equivariant bijection

\vspace{0.06in}
\begin{center} $\textup{exp}: \mathcal{N}_1(\mathfrak{g}) \xrightarrow{\sim} \mathcal{U}_1(G)$ \end{center}
\vspace{0.06in}

\noindent with the property that if $U$ is the unipotent radical of a parabolic subgroup $P \le G$ such that $U$ has nilpotence class less than $p$, then $\textup{exp}$ restricts to a $P$-equivariant isomorphism $\mathfrak{u} \xrightarrow{\sim} U$ of algebraic groups having tangent map equal to the identity on $\mathfrak{u}$.
\end{thm}

\begin{proof}

For all $X \in \mathcal{N}_1(\mathfrak{g})$, we can write $X = g \cdot Y$ for some $g \in G, Y \in \mathfrak{u}_J$.  We then define $\text{exp}(X) = g \cdot \varepsilon_{P_J}(Y)$.  By Proposition \ref{choice}, exp is well-defined since if $X = g_1 \cdot Y_1 = g_2 \cdot Y_2$, with $Y_1, Y_2 \in \mathfrak{u}_J$, then we have

\vspace{0.06in}
\begin{center} $\varepsilon_{P_J}(Y_1) = \varepsilon_{P_J}(g_1^{-1}g_2 \cdot Y_2) = g_1^{-1}g_2 \cdot \varepsilon_{P_J}(Y_2),$ \end{center} 
\vspace{0.06in}

\noindent so that $g_1 \cdot \varepsilon_{P_J}(Y_1) = g_2 \cdot \varepsilon_{P_J}(Y_2)$.  The uniqueness of $\text{exp}$ follows from Propositions \ref{Serre} and \ref{choice}.  To see that exp is a bijection, we note that the argument used in Proposition \ref{choice} applies equally well to the inverse map $\varepsilon_{P_J}^{-1}$, thus it follows that there is a well-defined map $\text{log}: \mathcal{U}_1(G) \rightarrow \mathcal{N}_1(\mathfrak{g})$ which is inverse to exp.
\end{proof}

In general, the map above will not be an expoential map in the sense of Definition \ref{map}, and we must make further assumptions on the prime $p$.  We thank the referee for pointing us in the right direction, namely that it is sufficient for us to require that $p$ be pretty good for $G$.

\begin{thm}
Let $G$ be a connected reductive group in pretty good characteristic.  Then $\textup{exp}$ is an exponential map for $G$.  Hence the map $$\textup{exp}_*:\mathcal{C}_r(\mathcal{N}_1(\mathfrak{g})) \rightarrow \textup{Hom}_{grp/k}(\mathbb{G}_{a(r)},G),$$ with $$\textup{exp}_*(X_0,\cdots,X_{r-1}) = \textup{exp}_{X_0}\cdots \textup{exp}_{X_{r-1}}^{(r-1)} \mid_{\mathbb{G}_{a(r)}},$$ is a bijection.
\end{thm}

\begin{proof}
It is clear from the construction of $\text{exp}$ and Proposition \ref{Serre} that $\text{exp}$ always satisfies condition (1) of Definition \ref{map}, so we are left to check condition (2).  Let $X \in \mathcal{N}_1(\mathfrak{g})$, and let $H$ be the closed subgroup scheme of $G$ isomorphic to $\mathbb{G}_{a(1)}$ which is determined by $X$ (see beginning of Section 2).  We have that $\mathfrak{h}=Lie(H)$ is spanned by $X$.  According to \cite[Theorem 3.3]{H}, since $p$ is a pretty good prime for $G$ and $H$ is a closed subgroup scheme, the scheme-theoretic centralizer of $H$, $C_G(H)$, is reduced.  It then follows from \cite[Lemma 3.1(ii)]{H} that $\mathfrak{h}$ is \textit{separable in} $\mathfrak{g}$, which means that $\text{dim} \, C_G(\mathfrak{h}) = \text{dim}_k \, C_{\mathfrak{g}}(\mathfrak{h})$, hence that $Lie(C_G(\mathfrak{h})) = C_{\mathfrak{g}}(\mathfrak{h})$.

Since $\text{exp}$ is $G$-equivariant, we have then that for any $0 \ne s \in k$, $C_G(\mathfrak{h})=C_G(sX)=C_G(\text{exp}(sX))$.  It follows that $\text{exp}_X(\mathbb{G}_a)$ acts trivially in the adjoint action on $C_{\mathfrak{g}}(X)$.  This proves that in pretty good characteristic $\text{exp}$ satisfies property (2) of Definition \ref{map}.
\end{proof}

\begin{remark}
See \cite[Remark 5.4]{H} for a nice description of reductive groups in pretty good characteristic.
\end{remark}

We conclude by showing that an exponential map for a connected reductive group is also an exponential map for its derived subgroup.

\begin{prop}
Let $G$ be a connected reductive group with derived subgroup $G^{\prime}$.  If $\mathcal{E}$ is an exponential map for $G$, then it is one for $G^{\prime}$. 
\end{prop}

\begin{proof}
As noted earlier in this section, we can identify $\mathcal{N}_1(\mathfrak{g}^{\prime})$ and $\mathcal{N}_1(\mathfrak{g})$.  For each $X \in \mathcal{N}_1(\mathfrak{g}^{\prime})$, since $\mathcal{E}(X)$ lies in a one-parameter subgroup of $G$, it also lies in a maximal unipotent subgroup of $G$, thus is in $G^{\prime}$.  Thus, $\mathcal{E}$ is a map from $\mathcal{N}_1(\mathfrak{g}^{\prime})$ to $G^{\prime}$, and if it is $G$-equivariant, then it is clearly $G^{\prime}$-equivariant.  Further, $C_{\mathfrak{g}^{\prime}}(X) = C_{\mathfrak{g}}(X) \cap \mathfrak{g}^{\prime}$, so if $\mathcal{E}_X(\mathbb{G}_a)$ acts trivially in the adjoint action on $C_{\mathfrak{g}}(X)$, it must also act trivially on $C_{\mathfrak{g}^{\prime}}(X)$.  Therefore $\mathcal{E}$ is an exponential map for $G^{\prime}$.
\end{proof}

\begin{remark}
For example, if $p$ is any prime, then $p$ is pretty good for $GL_n$ for all $n$, but $p$ is not pretty good for $SL_p$.  However, the truncated exponential series on nilpotent matrices is an exponential map not only for $GL_p$, but also for $SL_p$.
\end{remark}

\section{Saturation}

In this section we recall the ``saturation problem" for $p$-unipotent elements in $G$, which was introduced by Serre in \cite{Ser}, and then investigated successfully by Seitz in \cite{Sei}.  We show that the map $\text{exp}$ from the previous section provides a solution to this problem, and that this is the same answer as that already given in \cite{Sei}.

The problem is to associate to all $g \in \mathcal{U}_1(G)$ a canonical one-parameter subgroup $\phi_g$ of $G$ with the property that $\phi_g(1)=g$.  In \cite[\S 4]{Ser} it was shown that for $GL_n$ the saturation problem is solved by assigning to each $g \in \mathcal{U}_1(GL_n)$ the one-parameter subgroup $\phi_g$ defined by $\phi_g(s) = \text{exp}_p(s (\text{log}_p(g)))$, where

\vspace{0.05in}
\begin{center}$\text{log}_p(g) = (g-1) - \frac{(g-1)^2}{2} + \frac{(g-1)^3}{3} + \cdots + \frac{(-1)^p(g-1)^{p-1}}{p-1}$,\end{center}
\vspace{0.05in}

\noindent and for $X \in \mathcal{N}_1(\mathfrak{gl}_n)$

\vspace{0.05in}
\begin{center} $\text{exp}_p(X) = 1 + X + \frac{X^2}{2} + \cdots + \frac{X^{p-1}}{(p-1)!}$ \end{center}
\vspace{0.05in}

For $GL_n$, the map exp in Theorem \ref{exp} corresponds precisely to the truncated exponential series.  This is not surprising, and can be shown by Proposition \ref{oursgood} (below) together with the remarks following \cite[Theorem 1.3]{Sei}.  We immediately have the following generalization:

\begin{prop}\label{sat}
If $G$ is a connected reductive group and $p$ is good for $G$, then every $p$-unipotent element $g \in G$ lies in a canonical one-parameter subgroup $\phi_g$, where

\vspace{0.05in}
\begin{center}$\phi_g(s) = \textup{exp}(s(\textup{exp}^{-1}(g))), \text{ for all } s \in k$.\end{center}
\vspace{0.05in}

\end{prop}

\bigskip
In order to compare this solution to that given in \cite{Sei}, we must first recall the definition of a \textit{good $A_1$ subgroup}.  A closed subgroup $A \le G$ is of type $A_1$ if $A$ is isomorphic to $SL_2$ or $PSL_2$.  Let $T_A$ be a maximal torus of $A$.  We say that $A$ is good if $\mathfrak{g}$, as a $T_A$-module, has weights which are $\le 2p-2$.  Good $A_1$ subgroups were used by Seitz to specify the canonical one-parameter subgroups which contain $p$-unipotent elements.  Specifically, he proved the following:

\begin{thm}\label{mono}\cite{Sei}
There is a unique monomorphism $\psi_g: \mathbb{G}_a \rightarrow G$ with image contained in a good $A_1$ and satisfying $\psi_g(1) = g$.
\end{thm}

We are now in position to prove:

\begin{prop}\label{oursgood}
The one-parameter subgroups $\phi_g$ and $\psi_g$ agree for every $g \in \mathcal{U}_1(G)$.
\end{prop}

\begin{proof}
First, by the definition of $\psi_g$ given in \cite{Sei} we may assume that it comes from a homomorphism $\psi: SL_2 \rightarrow G$ such that $$\psi_g(a) = \psi \left(\begin{pmatrix} 1 & a \\ 0 & 1\end{pmatrix} \right).$$  Let $T_A$ be the image of the diagonal subgroup of $SL_2$, and for each $c \in k^{\times}$ we will use the notation $$\psi_T(c) = \psi \left(\begin{pmatrix} c & 0 \\ 0 & c^{-1}\end{pmatrix} \right).$$  Let $X_0$ be the image of $u_0 \in Dist(\mathbb{G}_a)$ under $d\psi_g$, the differential of the one-parameter subgroup $\psi_g$.  We see that in the adjoint action of $G$, $X_0$ is a weight vector for $T_A \le G$ with weight $2$.  Moreover, $d\psi_g(u_i)$ is a weight vector of weight $2p^i$ for $T_A \le G$ acting on $Dist(G)$.

Since $\text{exp}$ is $G$-equivariant, we have for each $c \in k^{\times}$ and $a \in k$ that $$\psi_T(c) \text{exp}_{X_0}(a) \psi_T(c^{-1}) = \text{exp}_{X_0}(c^2a).$$  It follows that each element $d\text{exp}_{X_0}(u_i) \in Dist(G)$ is also a weight vector for $T_A$ of weight $2p^i$.  By Lemma \ref{equal}, $d\psi_g(u_1) - d\text{exp}_{X_0}(u_1)$ is an element of $\mathfrak{g}$, and by preceeding remarks is a weight vector of $T_A$ of weight $2p$.  But all non-zero weight vectors of $T_A$ on $\mathfrak{g}$ are $\le 2p-2$, thus $d\psi_g(u_1) = d\text{exp}_{X_0}(u_1)$.  Continuing in this way we have that $d\psi_g(u_i) = d\text{exp}_{X_0}(u_i)$ for all $i$, from which is follows that $\psi_g = \text{exp}_{X_0}$.  As $\text{exp}_{X_0}(1) = g$, we have that $\phi_g = \text{exp}_{X_0}$, completing the proof.

\end{proof}


\bigskip
\noindent \textbf{Acknowledgements:} We wish to thank Eric Friedlander for useful remarks on an earlier version of this manuscript.  We also wish to thank the referee for many helpful observations.  This research was partially supported by grants from the Australian Research Council (DP1095831, DP0986774 and DP120101942).

\vspace{.2 in}
\noindent\tiny{DEPARTMENT OF MATHEMATICS \& STATISTICS, UNIVERSITY OF MELBOURNE, PARKVILLE, VIC 3010, AUSTRALIA}\\
paul.sobaje@unimelb.edu.au\\
Phone: \text{+}61 \, 401769982

\end{document}